\documentclass[reqno]{amsart}

\usepackage{amssymb}
\usepackage{graphicx}
\usepackage{amscd}
\usepackage[hidelinks]{hyperref}
\usepackage{color}
\usepackage{float}
\usepackage{graphics,amsmath,amssymb}
\usepackage{amsthm}
\usepackage{amsfonts}
\usepackage{latexsym}
\usepackage{epsf}
\usepackage{xifthen}
\usepackage{mathrsfs}
\usepackage{dsfont}
\usepackage{makecell}
\usepackage[FIGTOPCAP]{subfigure}
\usepackage{amsmath}
\allowdisplaybreaks[4]
\usepackage{listings}
\usepackage{etoolbox}
\usepackage{fancyhdr}
\usepackage{pdflscape}
\usepackage[title,toc,titletoc]{appendix}
\usepackage{enumitem}
\usepackage[noadjust]{cite}
\usepackage{forest}

 \newtheoremstyle{mytheorem}
 {3pt}
 {3pt}
 {\slshape}
 {}
 {\bfseries}
 {.}
 { }
 {}

\numberwithin{equation}{section}

\theoremstyle{theorem}
\newtheorem{theorem}{Theorem}[section]
\newtheorem*{theorem*}{Theorem}
\newtheorem{corollary}[theorem]{Corollary}
\newtheorem{lemma}[theorem]{Lemma}

\newtheorem{innercustomgeneric}{\customgenericname}
\providecommand{\customgenericname}{}
\newcommand{\newcustomtheorem}[2]{%
	\newenvironment{#1}[1]
	{%
		\renewcommand\customgenericname{#2}%
		\renewcommand\theinnercustomgeneric{##1}%
		\innercustomgeneric
	}
	{\endinnercustomgeneric}
}
\newcustomtheorem{ctheorem}{Theorem}

\theoremstyle{definition}
\newtheorem{definition}{Definition}[section]

\newtheorem*{example*}{Example}

\theoremstyle{remark}

\newtheorem*{remark*}{Remark}
\newtheorem*{remarks*}{Remarks}

\newtheoremstyle{named}{}{}{\itshape}{}{\bfseries}{.}{.5em}{#1\thmnote{ #3}}
\theoremstyle{named}

\newcommand{\Keywords}[1]{\ifthenelse{\isempty{#1}}{}{\smallskip \smallskip \noindent \textbf{Keywords}. #1}}
\newcommand{\MSC}[2][2010]{\ifthenelse{\isempty{#2}}{}{\smallskip \smallskip \noindent \textbf{#1MSC}. #2}}
\newcommand{\abstractnote}[1]{\ifthenelse{\isempty{#1}}{}{\smallskip \smallskip \noindent \textsuperscript{\dag}#1}}

\makeatletter
\def\specialsection{\@startsection{section}{1}%
  \z@{\linespacing\@plus\linespacing}{.5\linespacing}%
  {\normalfont}}
\def\section{\@startsection{section}{1}%
  \z@{.7\linespacing\@plus\linespacing}{.5\linespacing}%
  {\normalfont\scshape}}
\patchcmd{\@settitle}{\uppercasenonmath\@title}{\Large\boldmath}{}{}
\patchcmd{\@settitle}{\begin{center}}{\begin{flushleft}}{}{}
\patchcmd{\@settitle}{\end{center}}{\end{flushleft}}{}{}
\patchcmd{\@setauthors}{\MakeUppercase}{\normalsize}{}{}
\patchcmd{\@setauthors}{\centering}{\raggedright}{}{}
\patchcmd{\section}{\scshape}{\large\bfseries\boldmath}{}{}
\patchcmd{\subsection}{\bfseries}{\bfseries\boldmath}{}{}
\renewcommand{\@secnumfont}{\bfseries}
\patchcmd{\@startsection}{\@afterindenttrue}{\@afterindentfalse}{}{}
\patchcmd{\abstract}{\leftmargin3pc}{\leftmargin1pc}{}{}

\def\maketitle{\par
  \@topnum\z@ 
  \@setcopyright
  \thispagestyle{empty}
  \ifx\@empty\shortauthors \let\shortauthors\shorttitle
  \else \andify\shortauthors
  \fi
  \@maketitle@hook
  \begingroup
  \@maketitle
  \toks@\@xp{\shortauthors}\@temptokena\@xp{\shorttitle}%
  \toks4{\def\\{ \ignorespaces}}
  \edef\@tempa{%
    \@nx\markboth{\the\toks4
      \@nx\MakeUppercase{\the\toks@}}{\the\@temptokena}}%
  \@tempa
  \endgroup
  \c@footnote\z@
  \@cleartopmattertags
}
\makeatother

\newcommand{\cL}{\mathcal{L}}

\newcommand{\mA}{\mathscr{A}}

\newcommand{\mI}{\mathscr{I}}
\newcommand{\mM}{\mathscr{M}}
\newcommand{\mS}{\mathscr{S}}

\newcommand{\mW}{\mathscr{W}}

\newcommand{\ta}{\tilde{a}}

\newcommand{\diag}{\operatorname{diag}}

\title{Linked partition ideals and the Alladi--Schur theorem}

\author[G. E. Andrews]{George E. Andrews}
\address[G. E. Andrews]{Department of Mathematics, The Pennsylvania State University, University Park, PA 16802, USA}
\email{gea1@psu.edu}

\author[S. Chern]{Shane Chern}
\address[S. Chern]{Department of Mathematics and Statistics, Dalhousie University, Halifax, Nova Scotia, B3H 4R2, Canada}
\email{chenxiaohang92@gmail.com}

\author[Z. Li]{Zhitai Li}
\address[Z. Li]{Department of Mathematics, The Pennsylvania State University, University Park, PA 16802, USA}
\email{zfl5082@psu.edu}

\date{}

\begin{document}

\maketitle

\begin{abstract}

Let $\mathscr{S}$ denote the set of integer partitions into parts that differ by at least $3$, with the added constraint that no two consecutive multiples of $3$ occur as parts. We derive trivariate generating functions of Andrews--Gordon type for partitions in $\mathscr{S}$ with both the number of parts and the number of even parts counted. In particular, we provide an analytic counterpart of Andrews' recent refinement of the Alladi--Schur theorem.

\Keywords{Linked partition ideals, Alladi--Schur theorem, Andrews--Gordon type series, generating function.}

\MSC{11P84, 05A17.}
\end{abstract}

\section{Introduction}

The following result was proved by Schur \cite{Sch1926} in 1926.

\begin{ctheorem}{S}\label{th:S}
	Let $A(n)$ denote the number of partitions of $n$ into parts congruent to $\pm 1$ modulo $6$. Let $B(n)$ denote the number of partitions of $n$ into distinct nonmultiples of $3$. Let $D(n)$ denote the number of partitions of $n$ of the form $\mu_1+\mu_2+\cdots +\mu_s$ where $\mu_i-\mu_{i+1}\ge 3$ with strict inequality if $3\mid \mu_i$. Then
	$$A(n)=B(n)=D(n).$$
\end{ctheorem}

\noindent Alladi later made an observation (cf.~\cite[p.~46, (1.3)]{And2000}).

\begin{ctheorem}{AS}\label{th:AS}
	If we define $C(n)$ to be the number of partitions of $n$ into odd parts with none appearing more than twice, then
	$$C(n)=D(n).$$
\end{ctheorem}

\noindent On the other hand, Glei{\ss}berg \cite{Gle1928} had a refinement of Theorem \ref{th:S}.

\begin{ctheorem}{G}\label{th:G}
	Let $B(m,n)$ denote the number of partitions of $n$ into $m$ distinct nonmultiples of $3$. Let $D(m,n)$ denote the number of partitions of $n$ enumerated by $D(n)$ such that the total number of parts plus the number of multiples of $3$ among the parts equals $m$. Then
	$$B(m,n)=D(m,n).$$
\end{ctheorem}

\noindent Motivated by Glei{\ss}berg's result, Andrews \cite{And2017,And2019} also refined Theorem \ref{th:AS} in the same manner.

\begin{ctheorem}{A}\label{th:A}
	Let $C(m,n)$ denote the number of partitions of $n$ into $m$ odd parts with none appearing more than twice. Let $D'(m,n)$ denote the number of partitions of $n$ enumerated by $D(n)$ such that the total number of parts plus the number of even parts equals $m$. Then
	$$C(m,n)=D'(m,n).$$
\end{ctheorem}

\noindent A surprising point regarding Theorem \ref{th:A} is that partitions enumerated by $D(n)$ are controlled by a difference condition of size $3$ (so in terms of linked partition ideals, which will be explained in the next section, the modulus is $3$), and therefore less attention was paid to even parts in the previous studies.

Recall also that an \textit{Andrews--Gordon type series} is of the form
\begin{equation}\label{eq:And-conj-0}
\sum_{n_1,\ldots,n_r\ge 0}\frac{(-1)^{L_1(n_1,\ldots,n_r)}q^{Q(n_1,\ldots,n_r)+L_2(n_1,\ldots,n_r)}}{(q^{A_1};q^{A_1})_{n_1}\cdots (q^{A_r};q^{A_r})_{n_r}},
\end{equation}
in which $L_1$ and $L_2$ are linear forms and $Q$ is a quadratic form in $n_1,\ldots,n_r$, and the $q$-Pochhammer symbol is defined for $n\in\mathbb{N}\cup\{\infty\}$,
$$(A;q)_n:=\prod_{k=0}^{n-1} (1-A q^k).$$
This terminology comes from the multiple-folder generalization of the Rogers--Ramanujan identities due to Andrews \cite{And1974} and Gordon \cite{Gor1961}: for $1\le i\le k$ and $k\ge 2$,
\begin{equation*}
\prod_{\substack{n\ge 1\\n\not\equiv 0,\pm i\, (\operatorname{mod}\, 2k+1)}}\frac{1}{1-q^n}=\sum_{n_1,\ldots,n_{k-1}\ge 0}\frac{q^{N_1^2+N_2^2+\cdots N_{k-1}^2+N_i+N_{i+1}+\cdots +N_{k-1}}}{(q;q)_{n_1}(q;q)_{n_2}\cdots (q;q)_{n_{k-1}}}
\end{equation*}
where $N_j=n_j+n_{j+1}+\cdots +n_{k-1}$.

Let $\mS$ denote the set of partitions enumerated by $D(n)$ for all nonnegative $n$.

Further, for any partition $\lambda$, we denote by $|\lambda|$ the sum of all parts in $\lambda$, by $\sharp(\lambda)$ the number of parts in $\lambda$, and by $\sharp_{a,M}(\lambda)$ the number of parts in $\lambda$ that are congruent to $a$ modulo $M$.

Many discoveries have been made on the generating function of Andrews--Gordon type for the partition set $\mS$. In particular,
Andrews, Bringmann and Mahlburg \cite{ABM2015} proved that
\begin{align*}
\sum_{\lambda\in\mS}x^{\sharp(\lambda)}q^{|\lambda|}=\sum_{m,n\ge 0}\frac{(-1)^n q^{(m+3n)^2+\frac{m(m-1)}{2}}x^{m+2n}}{(q;q)_m(q^6;q^6)_n},
\end{align*}
Kur\c{s}ung\"{o}z \cite{Kur2021} proved that
\begin{align*}
&\sum_{\lambda\in\mS}x^{\sharp(\lambda)}q^{|\lambda|}\\
&\quad=\sum_{\substack{n_1\ge 0\\n_{21},n_{22}\ge 0}}\frac{q^{2n_1^2-n_1+6n_{21}^2-n_{21}+6n_{22}^2+n_{22}+6n_1(n_{21}+n_{22})+12n_{21}n_{22}}x^{n_1+2n_{21}+2n_{22}}}{(q;q)_{n_1}(q^6;q^6)_{n_{21}}(q^6;q^6)_{n_{22}}},
\end{align*}
and Alladi and Gordon \cite{AG1995} proved that
\begin{align}\label{eq:S-3-1}
&\sum_{\lambda\in\mS}x^{\sharp(\lambda)}y^{\sharp_{0,3}(\lambda)}q^{|\lambda|}\notag\\
&\quad=\sum_{n_1,n_2,n_3\ge 0}\frac{q^{3\binom{n_1}{2}+3\binom{n_2}{2}+6\binom{n_3}{2}+3n_1n_2+3n_2n_3+3n_3n_1+ n_1+2 n_2+3n_3}x^{n_1+n_2+n_3}y^{n_3}}{(q^3;q^3)_{n_1}(q^3;q^3)_{n_2}(q^3;q^3)_{n_3}}.
\end{align}
Recently, using the framework of linked partition ideals, Chern \cite{Che2021} reproduced \eqref{eq:S-3-1} and further proved that if $\mS_S$ denotes the set of partitions in $\mS$ whose smallest part is not in the set $S$ of positive integers, then
\begin{align}
&\sum_{\lambda\in\mS_{\{1\}}}x^{\sharp(\lambda)}y^{\sharp_{0,3}(\lambda)}q^{|\lambda|}\notag\\
&\quad=\sum_{n_1,n_2,n_3\ge 0}\frac{q^{3\binom{n_1}{2}+3\binom{n_2}{2}+6\binom{n_3}{2}+3n_1n_2+3n_2n_3+3n_3n_1+ 4n_1+2 n_2+3n_3}x^{n_1+n_2+n_3}y^{n_3}}{(q^3;q^3)_{n_1}(q^3;q^3)_{n_2}(q^3;q^3)_{n_3}}\label{eq:S-3-2}\\
\intertext{and}
&\sum_{\lambda\in\mS_{\{1,2,3\}}}x^{\sharp(\lambda)}y^{\sharp_{0,3}(\lambda)}q^{|\lambda|}\notag\\
&\quad=\sum_{n_1,n_2,n_3\ge 0}\frac{q^{3\binom{n_1}{2}+3\binom{n_2}{2}+6\binom{n_3}{2}+3n_1n_2+3n_2n_3+3n_3n_1+ 4n_1+5 n_2+6n_3}x^{n_1+n_2+n_3}y^{n_3}}{(q^3;q^3)_{n_1}(q^3;q^3)_{n_2}(q^3;q^3)_{n_3}}.\label{eq:S-3-3}
\end{align}
Notice that \eqref{eq:S-3-1} yields an analytic counterpart of Theorem \ref{th:G} by setting $y=x$:
\begin{align}
&(-xq;q^3)_\infty (-xq^2;q^3)_\infty\notag\\
&\quad=\sum_{n_1,n_2,n_3\ge 0}\frac{q^{3\binom{n_1}{2}+3\binom{n_2}{2}+6\binom{n_3}{2}+3n_1n_2+3n_2n_3+3n_3n_1+ n_1+2 n_2+3n_3}x^{n_1+n_2+2n_3}}{(q^3;q^3)_{n_1}(q^3;q^3)_{n_2}(q^3;q^3)_{n_3}}.
\end{align}

With Theorem \ref{th:A} in mind, our object is the following result.

\begin{theorem}
	We have
	\begin{align}
	\sum_{\lambda\in\mS}x^{\sharp(\lambda)}y^{\sharp_{0,2}(\lambda)}q^{|\lambda|}&=\sum_{n_1,n_2,n_3\ge 0}\frac{(-1)^{n_3}x^{n_1+n_2+2n_3}y^{n_2+n_3}}{(q^2;q^2)_{n_1}(q^2;q^2)_{n_2}(q^6;q^6)_{n_3}}\notag\\
	&\times q^{4\binom{n_1}{2}+4\binom{n_2}{2}+18\binom{n_3}{2}+2n_1n_2+6n_2n_3+6n_3n_1+ n_1+2 n_2+9n_3},\label{eq:S-2-1}\\
	\sum_{\lambda\in\mS_{\{1\}}}x^{\sharp(\lambda)}y^{\sharp_{0,2}(\lambda)}q^{|\lambda|}&=\sum_{n_1,n_2,n_3\ge 0}\frac{(-1)^{n_3}x^{n_1+n_2+2n_3}y^{n_2+n_3}}{(q^2;q^2)_{n_1}(q^2;q^2)_{n_2}(q^6;q^6)_{n_3}}\notag\\
	&\times q^{4\binom{n_1}{2}+4\binom{n_2}{2}+18\binom{n_3}{2}+2n_1n_2+6n_2n_3+6n_3n_1+ 3n_1+2 n_2+9n_3},\label{eq:S-2-2}\\
	\sum_{\lambda\in\mS_{\{1,2,3\}}}x^{\sharp(\lambda)}y^{\sharp_{0,2}(\lambda)}q^{|\lambda|}&=\sum_{n_1,n_2,n_3\ge 0}\frac{(-1)^{n_3}x^{n_1+n_2+2n_3}y^{n_2+n_3}}{(q^2;q^2)_{n_1}(q^2;q^2)_{n_2}(q^6;q^6)_{n_3}}\notag\\
	&\times q^{4\binom{n_1}{2}+4\binom{n_2}{2}+18\binom{n_3}{2}+2n_1n_2+6n_2n_3+6n_3n_1+ 5n_1+4 n_2+15n_3}.\label{eq:S-2-3}
	\end{align}
\end{theorem}

In particular, \eqref{eq:S-2-1} implies a corollary as follows.

\begin{corollary}\label{coro:A}
	Theorem \ref{th:A} is true. Also, it has an analytic counterpart
	\begin{align}\label{eq:A-analytic}
	\prod_{n\ge 0}(1+x q^{2n+1}+x^2 q^{4n+2})&=\sum_{n_1,n_2,n_3\ge 0}\frac{(-1)^{n_3}x^{n_1+2n_2+3n_3}}{(q^2;q^2)_{n_1}(q^2;q^2)_{n_2}(q^6;q^6)_{n_3}}\notag\\
	&\times q^{4\binom{n_1}{2}+4\binom{n_2}{2}+18\binom{n_3}{2}+2n_1n_2+6n_2n_3+6n_3n_1+ n_1+2 n_2+9n_3}.
	\end{align}
\end{corollary}

Our paper is organized as follows. In Section \ref{sec:LPI}, we give a brief review of linked partition ideals and use this framework to construct a matrix equation for generating functions related to the partition set $\mS$. Then in Sections \ref{sec:q-diff}--\ref{sec:S-2-1}, we present a proof of \eqref{eq:S-2-1}. The basic idea comes from the work of Chern and Li \cite{CL2020}. For \eqref{eq:S-2-2} and \eqref{eq:S-2-3}, we may proceed in a similar way so the details will be omitted. Next, we prove Corollary \ref{coro:A} in Section \ref{sec:coro}. Here, we provide two proofs --- one is $q$-hypergeometric and the other is computer-assisted. Finally, we close this paper with a conclusion in Section \ref{sec:conclusion}.

\section{Linked partition ideals and a matrix equation}\label{sec:LPI}

\subsection{Span one linked partition ideals}

The concept of linked partition ideals was introduced by Andrews \cite{And1972,And1974b,And1975} in the 1970s for a general theory of partition identities. We refer the interested reader to \cite[Chapter 8]{And1976} for a detailed account. Recently, Chern and Li \cite{CL2020} and Chern \cite{Che2020} revisited a special type of linked partition ideals, called \textit{span one linked partition ideals}, in their study of Andrews--Gordon type series for several partition sets.

\begin{definition}
	Assume that we are given
	\begin{itemize}[leftmargin=*,align=left]
		\renewcommand{\labelitemi}{\scriptsize$\blacktriangleright$}
		
		\item a finite set $\Pi=\{\pi_1,\pi_2,\ldots,\pi_K\}$ of integer partitions with $\pi_1=\emptyset$, the empty partition,
		
		\item a \textit{map of linking sets}, $\cL:\Pi\to P(\Pi)$, the power set of $\Pi$, with especially, $\cL(\pi_1)=\cL(\emptyset)=\Pi$ and $\pi_1=\emptyset\in \cL(\pi_k)$ for any $1\le k\le K$,
		
		\item and a positive integer $T$, called the \textit{modulus}, which is greater than or equal to the largest part among all partitions in $\Pi$.
	\end{itemize}
	We say a \textit{span one linked partition ideal} $\mI=\mI(\langle\Pi,\cL\rangle,T)$ is the collection of all partitions of the form
	\begin{align}\label{eq:decomp}
	\lambda&=\phi^0(\lambda_0)\oplus \phi^T(\lambda_1)\oplus \cdots \oplus \phi^{NT}(\lambda_N)\oplus \phi^{(N+1)T}(\pi_1)\oplus \phi^{(N+2)T}(\pi_1)\oplus \cdots\notag\\
	&=\phi^0(\lambda_0)\oplus \phi^T(\lambda_1)\oplus \cdots \oplus \phi^{NT}(\lambda_N),
	\end{align}
	where $\lambda_i\in\cL(\lambda_{i-1})$ for each $i$ and $\lambda_N$ is not the empty partition. We also include in $\mI$ the empty partition, which corresponds to $\phi^{0}(\pi_1)\oplus \phi^{T}(\pi_1)\oplus \cdots$. Here for any two partitions $\mu$ and $\nu$, $\mu\oplus\nu$ gives a partition by collecting all parts in $\mu$ and $\nu$, and $\phi^m(\mu)$ gives a partition by adding $m$ to each part of $\mu$.
\end{definition}

It was shown in \cite{Che2021} that the partition set $\mS$ in Schur's theorem is the span one linked partition ideal $\mI(\langle\Pi',\cL'\rangle,3)$, where
$\Pi'=\{\varpi_1=\emptyset,\varpi_2=(1),\varpi_3=(2),\varpi_4=(3)\}$ and
\begin{equation*}
\left\{
\begin{aligned}
\cL'(\varpi_1)&=\{\varpi_1,\varpi_2,\varpi_3,\varpi_4\},\\
\cL'(\varpi_2)&=\{\varpi_1,\varpi_2,\varpi_3,\varpi_4\},\\
\cL'(\varpi_3)&=\{\varpi_1,\varpi_3,\varpi_4\},\\
\cL'(\varpi_4)&=\{\varpi_1\}.
\end{aligned}
\right.
\end{equation*}
However, this equivalence is not sufficient when even parts are taken into consideration. Therefore, our first goal is the following refinement.

\begin{lemma}
	$\mS$ is also the span one linked partition ideal $\mI(\langle\Pi,\cL\rangle,6)$, where
	$\Pi=\{\pi_1=\emptyset,\pi_2=(1),\pi_3=(2),\pi_4=(3),\pi_5=(4),\pi_6=(1+4),\pi_7=(5),\pi_8=(1+5),\pi_9=(2+5),\pi_{10}=(6),\pi_{11}=(1+6),\pi_{12}=(2+6)\}$ and
	\begin{equation*}
	\left\{
	\begin{aligned}
	\cL(\pi_1)=\cL(\pi_2)=\cdots=\cL(\pi_6)&=\{\pi_1,\pi_2,\ldots,\pi_{12}\},\\
	\cL(\pi_7)=\cL(\pi_8)=\cL(\pi_9)&=\{\pi_1,\pi_3,\pi_4,\pi_5,\pi_7,\pi_9,\pi_{10},\pi_{12}\},\\
	\cL(\pi_{10})=\cL(\pi_{11})=\cL(\pi_{12})&=\{\pi_1,\pi_5,\pi_7,\pi_{10}\}.
	\end{aligned}
	\right.
	\end{equation*}
\end{lemma}

\begin{proof}
	We decompose each partition in $\mS$ into blocks $B_0, B_1, \ldots$ such that all parts between $6i+1$ and $6i+6$ fall into block $B_i$. By the definition of $\mS$, we find that if we apply the operator $\phi^{-6i}$ to the block $B_i$, then it is among $\Pi$. If $\phi^{-6i}(B_i)$ is any of $\{\pi_1,\pi_2,\ldots,\pi_6\}$, then the largest part in block $B_i$ is at most $6i+4$, and therefore $\phi^{-6(i+1)}(B_{i+1})$ can be any among $\Pi$. If $\phi^{-6i}(B_i)$ is any of $\{\pi_7,\pi_8,\pi_9\}$, then the largest part in block $B_i$ is $6i+5$, and therefore the smallest part in block $B_{i+1}$ is at least $6(i+1)+2$, that is $\phi^{-6(i+1)}(B_{i+1})\not\in \{\pi_2,\pi_6,\pi_8,\pi_{11}\}$. If $\phi^{-6i}(B_i)$ is any of $\{\pi_{10},\pi_{11},\pi_{12}\}$, then the largest part in block $B_i$ is $6i+6$, and therefore the smallest part in block $B_{i+1}$ is at least $6(i+1)+4$, that is $\phi^{-6(i+1)}(B_{i+1})\not\in \{\pi_2,\pi_3,\pi_4,\pi_6,\pi_8,\pi_9,\pi_{11},\pi_{12}\}$. Hence, all partitions in $\mS$ are also in $\mI(\langle\Pi,\cL\rangle,6)$. Conversely, it is straightforward to verify that all partitions in $\mI(\langle\Pi,\cL\rangle,6)$ satisfy the difference conditions defined for $D(n)$, and are thus in $\mS$.
\end{proof}

\subsection{A matrix equation}

For any partition $\lambda\in \mS=\mI(\langle\Pi,\cL\rangle,6)$, we always decompose it as in \eqref{eq:decomp}. For each $1\le k\le 12$, we write
$$G_k(x):=\sum_{\substack{\lambda\in\mS\\\lambda_0=\pi_k}}x^{\sharp(\lambda)}y^{\sharp_{0,2}(\lambda)}q^{|\lambda|},$$
that is, $G_k(x)$ is the generating function for partitions whose first decomposed block is $\pi_k$. By the definition of span one linked partition ideals, we have
\begin{align}
G_k(x)=x^{\sharp(\pi_k)}y^{\sharp_{0,2}(\pi_k)}q^{|\pi_k|}\sum_{j:\pi_j\in\cL(\pi_k)} G_j(xq^6).
\end{align}
Therefore,
\begin{equation}
\begin{pmatrix}
G_1(x)\\
G_2(x)\\
\vdots\\
G_{12}(x)
\end{pmatrix}
=
\mW(x).\mA.
\begin{pmatrix}
G_1(xq^6)\\
G_2(xq^6)\\
\vdots\\
G_{12}(xq^6)
\end{pmatrix}
\end{equation}
where
$$\mW(x)=\diag(1,xq,xyq^2,xq^3,xyq^4,x^2yq^5,xq^5,x^2q^6,x^2yq^7,xyq^6,x^2yq^7,x^2y^2q^8)$$
and
$$\setcounter{MaxMatrixCols}{15}
\mA=\begin{pmatrix}
1 & 1 & 1 & 1 & 1 & 1 & 1 & 1 & 1 & 1 & 1 & 1\\
1 & 1 & 1 & 1 & 1 & 1 & 1 & 1 & 1 & 1 & 1 & 1\\
1 & 1 & 1 & 1 & 1 & 1 & 1 & 1 & 1 & 1 & 1 & 1\\
1 & 1 & 1 & 1 & 1 & 1 & 1 & 1 & 1 & 1 & 1 & 1\\
1 & 1 & 1 & 1 & 1 & 1 & 1 & 1 & 1 & 1 & 1 & 1\\
1 & 1 & 1 & 1 & 1 & 1 & 1 & 1 & 1 & 1 & 1 & 1\\
1 & 0 & 1 & 1 & 1 & 0 & 1 & 0 & 1 & 1 & 0 & 1\\
1 & 0 & 1 & 1 & 1 & 0 & 1 & 0 & 1 & 1 & 0 & 1\\
1 & 0 & 1 & 1 & 1 & 0 & 1 & 0 & 1 & 1 & 0 & 1\\
1 & 0 & 0 & 0 & 1 & 0 & 1 & 0 & 0 & 1 & 0 & 0\\
1 & 0 & 0 & 0 & 1 & 0 & 1 & 0 & 0 & 1 & 0 & 0\\
1 & 0 & 0 & 0 & 1 & 0 & 1 & 0 & 0 & 1 & 0 & 0
\end{pmatrix}.$$

We further write
\begin{align}\label{eq:F-G}
\begin{pmatrix}
F_1(x)\\
F_2(x)\\
\vdots\\
F_{12}(x)
\end{pmatrix}=\mA.\begin{pmatrix}
G_1(x)\\
G_2(x)\\
\vdots\\
G_{12}(x)
\end{pmatrix}.
\end{align}
Then
\begin{align*}
\begin{pmatrix}
F_1(x)\\
F_2(x)\\
\vdots\\
F_{12}(x)
\end{pmatrix} = \mA.\mW(x).\mA.
\begin{pmatrix}
G_1(xq^6)\\
G_2(xq^6)\\
\vdots\\
G_{12}(xq^6)
\end{pmatrix},
\end{align*}
and therefore,
\begin{align}\label{eq:F}
\begin{pmatrix}
F_1(x)\\
F_2(x)\\
\vdots\\
F_{12}(x)
\end{pmatrix}=\mA.\mW(x).\begin{pmatrix}
F_1(xq^6)\\
F_2(xq^6)\\
\vdots\\
F_{12}(xq^6)
\end{pmatrix}.
\end{align}

Also, \eqref{eq:F-G} implies that
\begin{align}
F_1(x)=F_2(x)=\cdots=F_6(x)&=:A_1(x),\\
F_7(x)=F_8(x)=F_9(x)&=:A_2(x),\\
F_{10}(x)=F_{11}(x)=F_{12}(x)&=:A_3(x).
\end{align}
We may rewrite \eqref{eq:F} as
\begin{equation}\label{eq:A}
\begin{pmatrix}
A_1(x)\\
A_2(x)\\
A_3(x)
\end{pmatrix} = 
\mM.
\begin{pmatrix}
A_1(xq^6)\\
A_2(xq^6)\\
A_3(xq^6)
\end{pmatrix}
\end{equation}
where
\begin{equation*}
\scalebox{0.875}{%
	$\mM = 
	\begin{pmatrix}
	1+xq+xyq^2+xq^3+xyq^4+x^2yq^5 & xq^5+x^2q^6+x^2yq^7 & xyq^6+x^2yq^7+x^2y^2q^8\\
	1+xyq^2+xq^3+xyq^4 & xq^5+x^2yq^7 & xyq^6+x^2y^2q^8\\
	1+xyq^4 & xq^5 & xyq^6\\
	\end{pmatrix}.$}
\end{equation*}

Finally, we observe that
\begin{align}
\sum_{\lambda\in\mS}x^{\sharp(\lambda)}y^{\sharp_{0,2}(\lambda)}q^{|\lambda|}&=\sum_{k\in\{1,2,3,\ldots,12\}}G_k(x)\notag\\
&=A_1(x),\label{eq:S-A-1}\\
\sum_{\lambda\in\mS_{\{1\}}}x^{\sharp(\lambda)}y^{\sharp_{0,2}(\lambda)}q^{|\lambda|}&=\sum_{k\in\{1,3,4,5,7,9,10,12\}}G_k(x)\notag\\
&=A_2(x),\label{eq:S-A-2}\\
\sum_{\lambda\in\mS_{\{1,2,3\}}}x^{\sharp(\lambda)}y^{\sharp_{0,2}(\lambda)}q^{|\lambda|}&=\sum_{k\in\{1,5,7,10\}}G_k(x)\notag\\
&=A_3(x).\label{eq:S-A-3}
\end{align}

\section{A $q$-difference equation}\label{sec:q-diff}

In Sections \ref{sec:q-diff}--\ref{sec:S-2-1}, we present a proof of \eqref{eq:S-2-1}, following the idea in \cite{CL2020}. We will omit the proofs of \eqref{eq:S-2-2} and \eqref{eq:S-2-3}. However, all related \textit{Mathematica} codes were uploaded at the following url:
{\small
	\begin{center}
		\url{https://github.com/shanechern/Alladi-Schur}
	\end{center}
}

In order to reduce matrix equations such as \eqref{eq:A} to $q$-difference equations of $A_1(x)$, $A_2(x)$ and $A_3(x)$, an algorithm was presented in \cite[Section 3]{CL2020}, following the idea in \cite[Lemma 8.10]{And1976}. So our first result is as follows.

\begin{theorem}\label{th:A-q-diff-eqn}
	Let $A_1(x)$ be as in \eqref{eq:A}. Then
	\begin{equation*}
	\scalebox{0.95}{%
	$\begin{aligned}
	0 &= \big[1+x (q^7+ yq^8)\big] A_1(x)\\
	&\quad- \big[1 + x (q + q^3 + q^5 + q^7 + yq^2 + yq^4 +  yq^6 + yq^8)\\
	&\quad\quad+ 
	x^2 (q^6 + q^8 + q^{10} + yq^5 + 2 yq^7 + 2 yq^9 + 2 yq^{11} + yq^{13} + y^2q^8 + y^2q^{10} + y^2q^{12})\\
	&\quad\quad+ 
	x^3 (yq^{12} + yq^{14} + y^2q^{13} + y^2q^{15})\big]A_1(xq^6)\\
	&\quad+\big[x^2 yq^{15} +x^3 (-q^{21}+ yq^{16}+ y^2q^{17}- y^3q^{24})\\
	&\quad\quad+x^4 (-q^{22}- yq^{23}+ y^2q^{30}- y^3q^{25}- y^4q^{26})\\
	&\quad\quad+x^5 (y^2q^{31} + y^3q^{32})\big]A_1(xq^{12}).
	\end{aligned}
	$}
	\end{equation*}
\end{theorem}

\begin{proof}
	This result follows from an implementation of the algorithm in \cite[Section 3]{CL2020}. Detailed calculations can be found in the \textit{Mathematica} notebook \texttt{AS-1.nb} uploaded at the url listed at the beginning of Section \ref{sec:q-diff}.
\end{proof}

Let
\begin{equation}
A_1(x)=\sum_{M\ge 0}a(M)x^M.
\end{equation}
Our next task is to deduce a recurrence for the coefficients $a(M)$.

\begin{corollary}\label{coro:a-rec}
	For any $M\ge 0$,
	\begin{equation*}
	\scalebox{0.9}{%
		$\begin{aligned}
		0 &= q^{12M}(y^2q^{31} + y^3q^{32})a(M)\\
		&\quad +q^{12(M+1)}(-q^{22}- yq^{23}+ y^2q^{30}- y^3q^{25}- y^4q^{26})a(M+1)\\
		&\quad+ \big[{-q^{6(M+2)}}(yq^{12} + yq^{14} + y^2q^{13} + y^2q^{15})\\
		&\quad\quad+q^{12(M+2)}(-q^{21}+ yq^{16}+ y^2q^{17}- y^3q^{24})\big]a(M+2)\\
		&\quad +\big[{-q^{6(M+3)}}(q^6 + q^8 + q^{10} + yq^5 + 2 yq^7 + 2 yq^9 + 2 yq^{11} + yq^{13} + y^2q^8 + y^2q^{10} + y^2q^{12})\\
		&\quad\quad +q^{12(M+3)}yq^{15}\big]a(M+3)\\
		&\quad+\big[(q^7+ yq^8)-q^{6(M+4)}(q + q^3 + q^5 + q^7 + yq^2 + yq^4 +  yq^6 + yq^8)\big]a(M+4)\\
		&\quad+\big[1-q^{6(M+5)}\big]a(M+5).
		\end{aligned}
		$}
	\end{equation*}
\end{corollary}

\begin{proof}
	For each $M\ge 0$, we equate the coefficients of $x^{M+5}$ on both sides of the $q$-difference equation in Theorem \ref{th:A-q-diff-eqn}, and thus arrive at the desired recurrence.
\end{proof}

\section{Guessing the Andrews--Gordon type series}\label{sec:guess}

Notice that if a series $A(x)\in\mathbb{C}[[y]][[q]][[x]]$ satisfies the $q$-difference equation in Theorem \ref{th:A-q-diff-eqn}, then it is uniquely determined by $A(0)$. Recall also that, by \eqref{eq:S-A-1}, $A_1(0)=1$. Therefore, we can compute the first several terms of $a(M)$ by the $q$-difference equation of $A_1(x)$:
\begin{align*}
a(0)&=1,\\
a(1)&=\frac{q (1+ yq)}{1-q^2},\\
a(2)&=\frac{q^5 (q - q^7 + y + yq^2 - yq^4 - yq^{10} + y^2q^3 - y^2q^9)}{(1-q^2)(1-q^4)(1-q^6)},\\
a(3)&=\frac{q^{12} (1+ yq) (q^3+y+ yq^2- yq^4+ yq^8+ y^2q^5)}{(1-q^2)(1-q^4)(1-q^6)},\\
a(4)&=\frac{(*\cdots*)}{(1-q^2)(1-q^4)(1-q^6)(1-q^8)(1-q^{12})}.
\end{align*}

To guess a suitable candidate of a Andrews--Gordon type series for $A_1(x)$, we require a process of brute force. From $a(1)$, it is natural to expect summations of the form:
$$\sum_{n_1\ge 0}\frac{q^{\boldsymbol{?}}x^{n_1}}{(q^2;q^2)_{n_1}} \qquad\text{and} \qquad \sum_{n_2\ge 0}\frac{q^{\boldsymbol{?}}x^{n_2}y^{n_2}}{(q^2;q^2)_{n_2}}.$$
From $a(2)$, it is also highly possible that an extra summation is needed:
$$\sum_{n_3\ge 0}\frac{(-1)^{\boldsymbol{?}}q^{\boldsymbol{?}}x^{2n_3}y^{n_2}}{(q^6;q^6)_{n_3}}.$$
Therefore, we may first try the candidate of a triple summation:
\begin{align*}
&\sum_{n_1,n_2,n_3\ge 0}\frac{(-1)^{n_3}x^{n_1+n_2+2n_3}y^{n_2+n_3}}{(q^2;q^2)_{n_1}(q^2;q^2)_{n_2}(q^6;q^6)_{n_3}}\notag\\
&\qquad\times q^{\alpha_{11}\binom{n_1}{2}+\alpha_{22}\binom{n_2}{2}+\alpha_{33}\binom{n_3}{2}+\alpha_{12}n_1n_2+\alpha_{23}n_2n_3+\alpha_{31}n_3n_1+ \beta_1 n_1+\beta_2 n_2+\beta_3 n_3}.
\end{align*}
From $a(1)$, we may compute that $\beta_1=1$ and $\beta_2=2$. From $a(2)$, we may further compute that $\alpha_{11}=4$ and $\alpha_{22}=4$, and single out candidates of $\alpha_{12}=2$ and $\beta_3=9$. Continuing this process, we finally arrive at the right-hand side of \eqref{eq:S-2-1}.

\section{Another recurrence}\label{sec:rec}

Let
\begin{align}
\sum_{M\ge 0}\ta(M)x^M&=\sum_{n_1,n_2,n_3\ge 0}\frac{(-1)^{n_3}x^{n_1+n_2+2n_3}y^{n_2+n_3}}{(q^2;q^2)_{n_1}(q^2;q^2)_{n_2}(q^6;q^6)_{n_3}}\notag\\
&\times q^{4\binom{n_1}{2}+4\binom{n_2}{2}+18\binom{n_3}{2}+2n_1n_2+6n_2n_3+6n_3n_1+ n_1+2 n_2+9n_3}.
\end{align}
Our task in this section is to find a recurrence for $\ta(M)$. To do so, we require the \textit{Mathematica} package \texttt{qMultiSum} implemented by Riese \cite{Rie2003} of Research Institute for Symbolic Computation (RISC) of Johannes Kepler University. This package can be downloaded at the website of RISC:
{\small
\begin{center}
	\url{https://www3.risc.jku.at/research/combinat/software/ergosum/index.html}
\end{center}
}

We first import this package:
\begin{lstlisting}[language=Mathematica]
<<RISC`qMultiSum`
\end{lstlisting}
Then the following recurrence can be computed automatically by calling the commands:
\begin{lstlisting}[language=Mathematica]
ClearAll[M, n1, n2, n3, y];
n1 = M - n2 - 2 n3;
summand = ((-1)^n3 q^(
4 Binomial[n1, 2] + 4 Binomial[n2, 2] + 18 Binomial[n3, 2] + 
2 n1 n2 + 6 n2 n3 + 6 n3 n1 + n1 + 2 n2 + 9 n3) y^(
n2 + n3))/(qPochhammer[q^2, q^2, n1] qPochhammer[q^2, q^2, 
n2] qPochhammer[q^6, q^6, n3]);
stru = qFindStructureSet[summand, {M}, {n2, n3}, {2}, {2, 2}, {2, 2}, 
qProtocol -> True]
rec = qFindRecurrence[summand, {M}, {n2, n3}, {2}, {2, 2}, {2, 2}, 
qProtocol -> True, StructSet -> stru[[1]]]
sumrec = qSumRecurrence[rec]
\end{lstlisting}

\begin{theorem}\label{th:ta-rec}
	For any $M\ge 0$,
	\begin{align*}
	0&=q^{12 M + 24} y^2 (1 + 2 yq + y^2q^2 + q^{6 M + 22} + yq^{6 M + 23} + y^2q^{6 M+ 24}) \ta(M)\\
	&\quad - q^{12 M + 27} (1 + yq) (1 + yq + y^2q^2 - y^2q^8 + y^3q^3 + y^4q^4\\
	&\quad\quad + q^{6 M + 22} +
	 y^2q^{6 M + 24} + y^4q^{6 M + 26}) \ta(M+1)\\
	&\quad-q^{6 M + 17} (y + yq^2 + 2 y^2q + 
	2 y^2q^3 + y^3q^2 + y^3q^4 - q^{6 M + 15} + q^{6 M + 21}\\
	&\quad\quad - 2 yq^{6 M + 16} + 
	2 yq^{6 M + 22} + yq^{6 M + 24} - yq^{12 M + 38}\\
	&\quad\quad - 3 y^2q^{6 M + 17} + 2 y^2q^{6 M + 23} + y^2q^{6 M + 25} - y^2q^{12 M + 39}\\
	&\quad\quad - 
	2 y^3q^{6 M + 18} + 2 y^3q^{6 M + 24} + y^3q^{6 M + 26} - y^3q^{12 M + 40}\\
	&\quad\quad - y^4q^{6 M + 19} + y^4q^{6 M + 25}) \ta(M+2)\\
	&\quad -q^{6 M + 17} (1 + q^2 + q^4) (1 + yq) (1 + yq + yq^3 + y^2q^2\\
	&\quad\quad + q^{6 M + 20} + yq^{6 M + 21} +  y^2q^{6 M + 22}) \ta(M+3)\\
	&\quad+(1 - q^{6 M + 24}) (1 + 2 yq + y^2q^2 + q^{6 M + 16} + yq^{6 M + 17} + y^2q^{6 M + 18}) \ta(M+4).
	\end{align*}
\end{theorem}

\section{Proof of (\ref{eq:S-2-1})}\label{sec:S-2-1}

By \eqref{eq:S-A-1}, we find that to prove \eqref{eq:S-2-1}, it suffices to show that for $M\ge 0$,
\begin{align}\label{eq:a-ta}
a(M)=\ta(M).
\end{align}
Recall from Corollary \ref{coro:a-rec} and Theorem \ref{th:ta-rec} that both $a(M)$ and $\ta(M)$ satisfy a recurrence. By a result of Kauers and Koutschan \cite{KK2009}, we know that so does $a(M)-\ta(M)$. To compute the recurrence for $a(M)-\ta(M)$, we require the \textit{Mathematica} package \texttt{qGeneratingFunctions} implemented by Koutschan. This package can also be downloaded at the website of RISC, the url of which is listed at the beginning of Section \ref{sec:rec}.

To begin with, we import this package:
\begin{lstlisting}[language=Mathematica]
<<RISC`qGeneratingFunctions`
\end{lstlisting}
Now assuming that \texttt{sumrec1 == 0} is the recurrence for $a(M)$ as in Corollary \ref{coro:a-rec} and that \texttt{sumrec2 == 0} is the recurrence for $\ta(M)$ as in Theorem \ref{th:ta-rec}, then the recurrence for $a(M)-\ta(M)$ can be computed by calling the commands:
\begin{lstlisting}[language=Mathematica]
ClearAll[M,y];
QREPlus[{sumrec1 == 0}, {sumrec2 == 0}, SUM[M]]
\end{lstlisting}
The result gives us a fifth-order recurrence.

\begin{theorem}
	Let $d(M):=a(M)-\ta(M)$. Then for $M\ge 5$,
	\begin{align*}
	q^{29}(1 - q^{6 M}) d(M) &= -q^{12 M} y^2 (1 + yq) d(M-5)\\
	&\quad + q^{12 M + 3} (1 + yq - y^2q^8 + y^3q^3 + y^4q^4) d(M-4)\\
	&\quad + q^{6 M + 9} (1 + yq) (yq^{14} + yq^{16}\\
	&\quad\quad +q^{6 M+5} - yq^{6 M} - yq^{6 M + 6} + y^2q^{6 M + 7}) d(M-3)\\
	&\quad + q^{6 M + 20} (q^3 + q^5 + q^7 + yq^2 + 2 yq^4 + 2 yq^6\\
	&\quad\quad + 2 yq^8 + yq^{10} + y^2q^5 + y^2q^7 + y^2q^9 - yq^{6 M})d(M-2)\\
	&\quad - q^{24} (1 + yq) (q^{12} - q^{6 M}\\
	&\quad\quad - q^{6 M + 2} - q^{6 M + 4} - q^{6 M + 6}) d(M-1).
	\end{align*}
\end{theorem}

Therefore, to prove \eqref{eq:a-ta} for any $M\ge 0$, it suffices to verify that it holds true for $0\le M\le 4$. This can be done by a straightforward calculation.

\section{Proof of Corollary \ref{coro:A}}\label{sec:coro}

Taking $y=x$ in \eqref{eq:S-2-1}, we see that Theorem \ref{th:A} is equivalent to \eqref{eq:A-analytic}:
\begin{align*}
\prod_{n\ge 0}(1+x q^{2n+1}+x^2 q^{4n+2})&=\sum_{n_1,n_2,n_3\ge 0}\frac{(-1)^{n_3}x^{n_1+2n_2+3n_3}}{(q^2;q^2)_{n_1}(q^2;q^2)_{n_2}(q^6;q^6)_{n_3}}\\
&\times q^{4\binom{n_1}{2}+4\binom{n_2}{2}+18\binom{n_3}{2}+2n_1n_2+6n_2n_3+6n_3n_1+ n_1+2 n_2+9n_3}.
\end{align*}
In this section, we give two proofs of this identity --- one is $q$-hypergeometric and the other is computer-assisted.

For convenience, we write
\begin{align}\label{eq:S-def}
S(x)=\sum_{M\ge 0}s(M)x^M&:=\sum_{n_1,n_2,n_3\ge 0}\frac{(-1)^{n_3}x^{n_1+2n_2+3n_3}}{(q^2;q^2)_{n_1}(q^2;q^2)_{n_2}(q^6;q^6)_{n_3}}\notag\\
&\;\times q^{4\binom{n_1}{2}+4\binom{n_2}{2}+18\binom{n_3}{2}+2n_1n_2+6n_2n_3+6n_3n_1+ n_1+2 n_2+9n_3}.
\end{align}

\subsection{A $q$-hypergeometric proof}
For convenience, we write for integers $\beta_1$, $\beta_2$ and $\beta_3$:
\begin{align*}
	\sigma(\beta_1,\beta_2,\beta_3)&:=\frac{(-1)^{n_3}x^{n_1+2n_2+3n_3}}{(q^2;q^2)_{n_1}(q^2;q^2)_{n_2}(q^6;q^6)_{n_3}}\\
	&\;\quad\times q^{4\binom{n_1}{2}+4\binom{n_2}{2}+18\binom{n_3}{2}+2n_1n_2+6n_2n_3+6n_3n_1+ \beta_1 n_1+\beta_2 n_2+\beta_3 n_3}
\end{align*}
and
\begin{align*}
\Sigma(\beta_1,\beta_2,\beta_3):=\sum_{n_1,n_2,n_3\ge 0}\sigma(\beta_1,\beta_2,\beta_3).
\end{align*}

Our proof relies on the following lemma.

\begin{lemma}\label{le:sigma}
	Let $k_1$, $k_2$ and $k_3$ be nonnegative integers. Then for any integers $\beta_1$, $\beta_2$ and $\beta_3$,
	\begin{align}\label{eq:sigma-1}
		&\sum_{n_1,n_2,n_3\ge 0}\sigma(\beta_1,\beta_2,\beta_3)(q^{2n_1};q^{-2})_{k_1}(q^{2n_2};q^{-2})_{k_2}(q^{6n_3};q^{-6})_{k_3}\notag\\
		&=(-1)^{k_3}x^{k_1+2k_2+3k_3}q^{4\binom{k_1}{2}+4\binom{k_2}{2}+18\binom{k_3}{2}+2k_1k_2+6k_2k_3+6k_3k_1+k_1\beta_1+k_2\beta_2+k_3\beta_3}\notag\\
		&\quad\times\Sigma(\beta_1+4k_1+2k_2+6k_3,\beta_2+2k_1+4k_2+6k_3,\beta_3+6k_1+6k_2+18k_3).
	\end{align}
\end{lemma}

\begin{proof}
	We have
	\begin{align*}
		&\sum_{n_1,n_2,n_3\ge 0}\sigma(\beta_1,\beta_2,\beta_3)(q^{2n_1};q^{-2})_{k_1}(q^{2n_2};q^{-2})_{k_2}(q^{6n_3};q^{-6})_{k_3}\\
		&=\sum_{n_1,n_2,n_3\ge 0}\frac{(-1)^{n_3}x^{n_1+2n_2+3n_3}(q^{2n_1};q^{-2})_{k_1}(q^{2n_2};q^{-2})_{k_2}(q^{6n_3};q^{-6})_{k_3}}{(q^2;q^2)_{n_1}(q^2;q^2)_{n_2}(q^6;q^6)_{n_3}}\\
		&\quad\times q^{4\binom{n_1}{2}+4\binom{n_2}{2}+18\binom{n_3}{2}+2n_1n_2+6n_2n_3+6n_3n_1+ \beta_1 n_1+\beta_2 n_2+\beta_3 n_3}\\
		&=\sum_{\substack{n_1\ge k_1\\n_2\ge k_2\\n_3\ge k_3}}\frac{(-1)^{n_3}x^{n_1+2n_2+3n_3}}{(q^2;q^2)_{n_1-k_1}(q^2;q^2)_{n_2-k_2}(q^6;q^6)_{n_3-k_3}}\\
		&\quad\times q^{4\binom{n_1}{2}+4\binom{n_2}{2}+18\binom{n_3}{2}+2n_1n_2+6n_2n_3+6n_3n_1+ \beta_1 n_1+\beta_2 n_2+\beta_3 n_3}\\
		&=\sum_{n_1,n_2,n_3\ge 0}\frac{(-1)^{n_3+k_3}x^{(n_1+k_1)+2(n_2+k_2)+3(n_3+k_3)}}{(q^2;q^2)_{n_1}(q^2;q^2)_{n_2}(q^6;q^6)_{n_3}}\\
		&\quad\times q^{4\binom{n_1+k_1}{2}+4\binom{n_2+k_2}{2}+18\binom{n_3+k_3}{2}+2(n_1+k_1)(n_2+k_2)+6(n_2+k_2)(n_3+k_3)+6(n_3+k_3)(n_1+k_1)}\\
		&\quad\times q^{\beta_1 (n_1+k_1)+\beta_2 (n_2+k_2)+\beta_3 (n_3+k_3)}.
	\end{align*}
	This gives \eqref{eq:sigma-1} after simplification.
\end{proof}

Now, we show the following result.
\begin{theorem}\label{eq:Sigma-1-2-9}
	We have
	\begin{align}
		\Sigma(1,2,9)=(1+xq+x^2q^2)\Sigma(3,6,15).
	\end{align}
\end{theorem}

\begin{proof}
	We split $1$ as
	\begin{align*}
		1=I_1+I_2+I_3+I_4+I_5+I_6+I_7+I_8+I_9,
	\end{align*}
	where
	\begin{align*}
	I_1&=q^{2n_1+4n_2+6n_3},\\
	I_2&=q^{-2n_1+2n_2+2}(1-q^{2n_1}),\\
	I_3&=1-q^{2n_2},\\
	I_4&=q^{4n_1+2n_2+6n_3}(1-q^{2n_2}),\\
	I_5&=-q^{-2n_1+2n_2+2}(1-q^{2n_1})(1-q^{2n_1-2}),\\
	I_6&=q^{2n_1+4n_2+6n_3-2}(1-q^{2n_1})(1-q^{2n_2}),\\
	I_7&=q^{2n_1+4n_2}(1-q^{6n_3}),\\
	I_8&=q^{2n_1+2n_2+6n_3}(1-q^{2n_1})(1-q^{2n_2})(1-q^{2n_2-2}),\\
	I_9&=q^{2n_1+2n_2}(1-q^{2n_2})(1-q^{6n_3}).
	\end{align*}
	Then by Lemma \ref{le:sigma},
	\begin{align*}
	\sum_{n_1,n_2,n_3\ge 0}\sigma(1,2,9)I_1&=\sum_{n_1,n_2,n_3\ge 0}\sigma(3,6,15)\\
	&=\Sigma(3,6,15),\\
	\sum_{n_1,n_2,n_3\ge 0}\sigma(1,2,9)I_2&=q^2\sum_{n_1,n_2,n_3\ge 0}\sigma(-1,4,9)(1-q^{2n_1})\\
	&=xq\Sigma(3,6,15),\\
	\sum_{n_1,n_2,n_3\ge 0}\sigma(1,2,9)I_3&=\sum_{n_1,n_2,n_3\ge 0}\sigma(1,2,9)(1-q^{2n_2})\\
	&=x^2q^2\Sigma(3,6,15),\\
	\sum_{n_1,n_2,n_3\ge 0}\sigma(1,2,9)I_4&=\sum_{n_1,n_2,n_3\ge 0}\sigma(5,4,15)(1-q^{2n_2})\\
	&=x^2q^4\Sigma(7,8,21),\\
	\sum_{n_1,n_2,n_3\ge 0}\sigma(1,2,9)I_5&=-q^2\sum_{n_1,n_2,n_3\ge 0}\sigma(-1,4,9)(1-q^{2n_1})(1-q^{2n_1-2})\\
	&=-x^2q^4\Sigma(7,8,21),\\
	\sum_{n_1,n_2,n_3\ge 0}\sigma(1,2,9)I_6&=q^{-2}\sum_{n_1,n_2,n_3\ge 0}\sigma(3,6,15)(1-q^{2n_1})(1-q^{2n_2})\\
	&=x^3q^9\Sigma(9,12,27),\\
	\sum_{n_1,n_2,n_3\ge 0}\sigma(1,2,9)I_7&=\sum_{n_1,n_2,n_3\ge 0}\sigma(3,6,9)(1-q^{6n_3})\\
	&=-x^3q^9\Sigma(9,12,27),\\
	\sum_{n_1,n_2,n_3\ge 0}\sigma(1,2,9)I_8&=\sum_{n_1,n_2,n_3\ge 0}\sigma(3,4,15)(1-q^{2n_1})(1-q^{2n_2})(1-q^{2n_2-2})\\
	&=x^5q^{19}\Sigma(11,14,33),\\
	\sum_{n_1,n_2,n_3\ge 0}\sigma(1,2,9)I_9&=\sum_{n_1,n_2,n_3\ge 0}\sigma(3,4,9)(1-q^{6n_3})\\
	&=-x^5q^{19}\Sigma(11,14,33).
	\end{align*}
	Therefore,
	\begin{align*}
	\Sigma(1,2,9)&=\sum_{n_1,n_2,n_3\ge 0}\sigma(1,2,9)(I_1+I_2+I_3+I_4+I_5+I_6+I_7+I_8+I_9)\\
	&=\Sigma(3,6,15)+xq\Sigma(3,6,15)+x^2q^2\Sigma(3,6,15),
	\end{align*}
	which is our desired result.
\end{proof}

\begin{proof}[First proof of \eqref{eq:A-analytic}]
	From \eqref{eq:S-def}, we have $S(x)=\Sigma(1,2,9)$ and $S(xq^2)=\Sigma(3,6,15)$. It follows from Theorem \ref{eq:Sigma-1-2-9} that
	$$S(x)=(1+xq+x^2q^2)S(xq^2).$$
	Finally, we conclude by $S(0)=1$ that
	$$S(x)=\prod_{n\ge 0}(1+x q^{2n+1}+x^2 q^{4n+2}),$$
	which is \eqref{eq:A-analytic}.
\end{proof}

\subsection{A computer-assisted proof}

Here, our task is to derive a recurrence for the coefficients $s(M)$ in \eqref{eq:S-def}, which is also computed by the \texttt{qMultiSum} package. We call the following commands:
\begin{lstlisting}[language=Mathematica]
ClearAll[M, n1, n2, n3];
n1 = M - 2 n2 - 3 n3;
summand = ((-1)^n3 q^(
4 Binomial[n1, 2] + 4 Binomial[n2, 2] + 18 Binomial[n3, 2] + 
2 n1 n2 + 6 n2 n3 + 6 n3 n1 + n1 + 2 n2 + 9 n3))/(qPochhammer[
q^2, q^2, n1] qPochhammer[q^2, q^2, n2] qPochhammer[q^6, q^6, 
n3]);
stru = qFindStructureSet[summand, {M}, {n2, n3}, {2}, {2, 1}, {1, 1}, 
qProtocol -> True]
rec = qFindRecurrence[summand, {M}, {n2, n3}, {2}, {2, 1}, {1, 1}, 
qProtocol -> True, StructSet -> stru[[1]]]
sumrec = qSumRecurrence[rec]
\end{lstlisting}
Then we arrive at a sixth-order recurrence.

\begin{theorem}\label{th:S-rec}
	For $M\ge 0$,
	\begin{align*}
	0&=q^{6 M + 24}s(M) -q^{4 M + 18} (1 + q^2 + q^4) s(M+2) -q^{6 M + 27} s(M+3)\\
	&\quad +q^{2 M + 10} (1 + q^2 + q^4 - q^{2 M + 12}) s(M+4) +q^{4 M + 21} s(M+5)\\
	&\quad - (1 - q^{2 M + 12}) s(M+6).
	\end{align*}
\end{theorem}

Now, we translate this recurrence to a $q$-difference equation for $S(x)$.

\begin{corollary}\label{coro:S-diff}
	We have
	\begin{align*}
	0&=S(x)-\big[x^2q^2(1+q^2+q^4)+1\big]S(xq^2)\\
	&\quad+\big[x^4q^{10}(1+q^2+q^4)+x^2q^6-xq\big]S(xq^4)-\big[x^6q^{24}-x^3q^9\big]S(xq^6).
	\end{align*}
\end{corollary}

\begin{proof}
	We multiply by $x^{M+6}$ on both sides of the recurrence in Theorem \ref{th:S-rec}, and sum over $M\ge 0$. Then
	{\small\begin{align*}
	0&=x^6q^{24}S(xq^6)-x^4q^{10}(1 + q^2 + q^4)\left(S(xq^4)-\sum_{k=0}^1 s(k)(xq^4)^k\right)\\
	&\quad - x^3q^9 \left(S(xq^6)-\sum_{k=0}^2 s(k)(xq^6)^k\right) +x^2q^2(1 + q^2 + q^4)\left(S(xq^2)-\sum_{k=0}^3 s(k)(xq^2)^k\right)\\
	&\quad-x^2q^6\left(S(xq^4)-\sum_{k=0}^3 s(k)(xq^4)^k\right) + xq\left(S(xq^4)-\sum_{k=0}^4 s(k)(xq^4)^k\right)\\
	&\quad -\left(S(x)-\sum_{k=0}^5 s(k)x^k\right) + \left(S(xq^2)-\sum_{k=0}^5 s(k)(xq^2)^k\right).
	\end{align*}}
	Inserting the values of $s(0)$, \ldots, $s(5)$ yields the desired result.
\end{proof}

\begin{proof}[Second proof of \eqref{eq:A-analytic}]
	Notice that if any $F(x)\in \mathbb{C}[[q]][[x]]$ satisfies the $q$-difference equation in Corollary \ref{coro:S-diff}, then it is uniquely determined by $F(0)$. Therefore, if we write
	$$P(x):=\prod_{n\ge 0}(1+x q^{2n+1}+x^2 q^{4n+2}),$$
	by noticing that $S(0)=P(0)=1$, it suffices to verify that
	\begin{align*}
	0&=P(x)-\big[x^2q^2(1+q^2+q^4)+1\big]P(xq^2)\\
	&\quad+\big[x^4q^{10}(1+q^2+q^4)+x^2q^6-xq\big]P(xq^4)-\big[x^6q^{24}-x^3q^9\big]P(xq^6).
	\end{align*}
	Namely,
	\begin{align*}
	0&=(1+xq+x^2q^2)(1+xq^3+x^2q^6)(1+xq^5+x^2q^{10})\\
	&\quad -\big[x^2q^2(1+q^2+q^4)+1\big](1+xq^3+x^2q^6)(1+xq^5+x^2q^{10})\\
	&\quad +\big[x^4q^{10}(1+q^2+q^4)+x^2q^6-xq\big](1+xq^5+x^2q^{10})\\
	&\quad -\big[x^6q^{24}-x^3q^9\big].
	\end{align*}
	This can be checked easily.
	
	Thus, we have $P(x)=S(x)$, which is exactly \eqref{eq:A-analytic}.
\end{proof}

\section{Conclusion}\label{sec:conclusion}

In light of the fact that Schur's theorem seems to be at the center of the theory of partition identities (perhaps even more so than the Rogers--Ramanujan identities), it is a major step forward to have this theorem now embedded in the theory of linked partition ideals. The point of paper \cite{And2017} was essentially to make the case that Alladi's addition to Schur's theorem is actually the core of this result, and identity \eqref{eq:A-analytic} certainly reaffirms that observation.

It is clear that the theory of linked partition ideals is still in its infancy after almost forty years. The goal would be a classification theorem for linked partition ideals comparable to the classification theorem \cite[Theorem 8.4, p.~126]{And1976} for partition ideals of order $1$. With the advent of powerful computer algebra packages, we have seen advances such as obtained in this paper and in \cite{Che2021} and \cite{CL2020}. We hope to see much progress in the future.

\subsection*{Acknowledgements}

The first author was supported by a grant (\#633284) from the Simons Foundation. The second author was supported by a Killam Postdoctoral Fellowship from the Killam Trusts.

\bibliographystyle{amsplain}

\end{document}